\documentclass[12pt]{article}%
\usepackage{graphicx}
\usepackage{multicol,multirow}
\usepackage{amsmath,amssymb,amsfonts}
\usepackage{mathrsfs}
\usepackage{rotating}
\usepackage{appendix}
\usepackage[numbers]{natbib}
\usepackage{amsthm}
\usepackage{lipsum}

\newtheorem{theorem}{Theorem}[section]
\newtheorem{lemma}[theorem]{Lemma}
\newtheorem{corollary}[theorem]{Corollary}
\newtheorem{proposition}[theorem]{Proposition}
\newtheorem{definition}{Definition}[section]
\newtheorem{remark}[theorem]{Remark}
\newtheorem{example}[theorem]{Example}

\numberwithin{equation}{section}
\usepackage{amstext}
\usepackage{latexsym}
\usepackage{amssymb}
\usepackage{amsmath,amsthm,centernot}

\usepackage{graphicx}
\usepackage{ifpdf}

\usepackage[pagewise]{lineno}

\allowdisplaybreaks

\newcommand{\nUparrow}{\Uparrow\hspace{-.24cm}\diagup\hspace{.31cm}}
\newcommand{\nDownarrow}{\Downarrow\hspace{-.24cm}\diagup\hspace{.31cm}}

\def\R{\mathbb R}

\def\N{\mathbb N}
\def\Z{\mathbb Z}

\def\int{{\rm int\;}}

\def\eps{\varepsilon}
\def\PSL{{\rm PSL}}
\def\sus{{\rm sus}}

\newcommand{\SL}{{\rm SL}}
\newcommand{\fix}{{\rm fix}}
\def\diam{{\rm diam\,}}
\usepackage{hyperref}

\begin{document}

	\title{Katok-Hasselblatt-kinematic expansive flows}
	\author{Huynh Minh Hien\\
	Department of Mathematics and Statistics,\\
		{Quy Nhon University},\\ {170 An Duong Vuong}, {Quy Nhon},  {Vietnam}\\
	huynhminhhien@qnu.edu.vn} \date{}

\maketitle


	\begin{abstract} In this paper we introduce a new notion of expansive flows, which is 
		the combination of expansivity in the sense of Katok and Hasselblatt and kinematic expansivity,
		named KH-kinematic expansivity. We  present  new properties of several variations of expansivity.
		A new hierarchy of expansive flows is given.
	\end{abstract}

\textbf{Keywords}: {KH-kinematic expansivity, Expansive flows, Fixed points.}
	
\textbf{MSC}: {37B05}, {37C10, 37C25}.
	
\maketitle

\section{Introduction}

The concept of expansivity plays an important role in the study of
discrete and continuous dynamical systems. 
Expansive flows have been studying for almost a half century.
In 1972, Bowen and Walters \cite{bw} gave a definition of expansivity for flows  that is called `expansive in the sense of Bowen and Walters' (or shortly `BW-expansive'). Since then, 
there have been several variations of expansive flows introduced, depending on the kind of reparametrizations.
BW-expansivity is so-called $C$-expansivity (\cite{keynes})
since the reparametrizations are continuous functions in the set denoted by $C$.  
$K$-expansivity means that reparametrizations are in a set $K$ consisting of increasing homeomorphisms (\cite{komuro}). 
The definition of $K$-expansive flow is the same to that of expansive flows given by Flinn in his Ph.D thesis \cite{flinn}. 
Bowen and Walters \cite{bw} first showed  that for flows without fixed points,  $C$-expansivity and $K$-expansivity are equivalent. 
Then Oka \cite{oka90} proved that in general (i.e. with fixed points presented), the definitions of $C$-expansive and $K$-expansive flows are equivalent.
The class of $C$-expansive flows includes Anosov flows and suspensions of expansive homeomorphisms. 
In 1984, Komuro \cite{komuro} introduced the notion of $K^*$-expansive flows 
to study the geometric Lorenz attractor, which is not $K$-expansive. Then Oka \cite{oka90}
again showed that in the case of fixed-point-free flows,  $K$-expansivity is equivalent to $K^*$-expansivity.
In this paper, we prove this property for flows with open sets of fixed points. 

In 1984, Gura \cite{gura} introduced the term  `separating' for flows in order to study horocycle flows. The author showed that
the horocycle flow on compact surfaces with negative curvature is separating together with its time changes
(so-called strong separating). Recently, Huynh \cite{hien20}  proved
that the horocycle flow on compact surfaces of constant negative curvature is strong kinematic, which
is stronger than being strong separating.  
In 2016, a collection of new notions of expansive flows was presented by Artigue \cite{artigue16}, including
geometric expansivity, kinematic expansivity, geometric separation,
strong kinematic expansivity, strong separation,... A hierarchy of expansive flows is given 
with many counterexamples to analyse the relations of expansive properties.

Back in  1995, in a well-known book by Katok and Hasselblatt \cite{kh},
the authors discovered a new definition of expansive flows
by considering reparameterization in a particular way; this expansivity
is later called `KH-expansivity'.   Artigue \cite{artigue18}
then showed that a flow is KH-expansive 
if and only if it is separating and the set of fixed points is open.
Huynh \cite{hien19} proved that the horocycle flow on compact surfaces of constant negative curvature 
is KH-expansive. 
Roughly speaking,  a flow is called  KH-expansive if the orbit of a point
and the  orbit of another point, which is reparameterizated, are close enough forever then 
these points must be in the same orbit. It is natural to require
these points belong to an orbit segment with small time and this  
motivates to consider a new variation of expansivity. 
In this paper we introduce a new concept of expansive flows, 
which is the combination of KH-expansivity and kinematic expansivity, named \textit{Katok-Hasselblatt-kinematic expansivity} or \textit{KH-kinematic expansivity} for short.
The class of KH-kinematic expansive flows consists of kinematic expansive flows with open sets of fixed points 
but does not admit BW-expansivity. 
One example of KH-kinematic expansive flows is the horocycle flow on compact surfaces with constant negative curvature. 
Since KH-expansivity and KH-kinematic expansivity are not invariant properties under time change of flows, 
it is necessary to consider the notions of strong KH-expansivity and strong KH-kinematic expansivity. 
Hierarchy of expansive flows and counterexamples are given to analyse the relationships of the classes of expansive and separating flows\footnote{Expansive and separating flows are sometimes called in common \textit{expansive flows}.}. This may be seen as a supplement of hierarchy of expansive  flows provided by Artigue \cite{artigue16}. We also show that if the set of fixed points is open, 
then $C$-expansivity, $K$-expansivity and $K^*$-expansivity are equivalent (Theorem \ref{KK*}).

The paper is organized as follows. In Section \ref{s2}, we present some important results
of fixed points sets of flows which will be used in this paper. A main result of this section is 
that any suspension flow has no fixed points. Section \ref{s3} defines and states basic 
properties of expansive and separating flows in the versions of $C$-expansive, $K$-expansive, $K^*$-expansive, geometric expansive, kinematic expansive, strong
kinematic expansive, KH-expansive,  $C$-separating, geometric separating, strong separating and separating flows.
We show that a flow with open fixed points set is $K$-expansive if and only if it is $K^*$-expansive.
In Section \ref{s4} we introduce the definition and provide equivalent properties of KH-kinematic expansivity.
Strong KH-kinematic expansivity, KH-positive kinematic expansivity are also discovered.
Finally, we provide a hierarchy of expansive flows with counterexamples to analyse relations of concepts about expansive and separating flows.

Throughout the present paper, we denote by $C$ the set of continuous functions $s:\R\to\R$ with $s(0)=0$
and by $K$ the set of increasing homeomorphisms in $C$. We always assume that
$\phi:\R\times X\to X$ is a continuous flow on a given compact metric space $(X,d)$.

\section{Fixed points of flows}\label{s2}

The fixed points set of a flow plays an important role in forming property of the flow. 
This section is devoted to presenting  some important properties which will be used in the next sections.

\subsection{Fixed points and periodic points}

\begin{definition}\rm Let $\phi:X\to X$ be a flow. 
	
	(i) A point $x\in X$ is called a \textit{fixed point} of $\phi$ if $\phi_t(x)=x$ for all $t\in\R$.
	The set of fixed points of $\phi$ is denoted by $\fix(\phi)$. 
	
	(ii) A point $x\in X$ is called a \textit{periodic point} of $\phi$ if there exists $T>0$ such that
	$\phi_T(x)=x$ and $\phi_t(x)\ne x$ for some $t\in\R$. Such a  $T$ is called a period of $x$. 
\end{definition}
A fixed point is so-called a singular point. A point which is not a singular point is called a regular point. 
\begin{example}\rm
	Let us consider a flow $\phi$ on $\R$ defined by $\phi_t(x)=xe^t$ for all $t,x\in\R$.
	The point 0 is the unique fixed point of $\phi$. 
\end{example}
\begin{proposition}\label{fspn} Let $\phi$ be a continuous flow on a compact metric space $(X,d)$. 
	If $\fix(\phi)$ is open, then $\phi$ does not have periodic points with arbitrarily small periods.
\end{proposition}
\begin{proof} 	Suppose in contrary that $\phi$ has arbitrarily small periods. 
	Let  $x_n\in \widetilde X=X\setminus \fix(\phi)$ and $t_n\to 0^+$ be such that $\phi_{t_n}(x_n)=x_n$ 
	for $n\in\N$. By hypothesis,  $\widetilde X$ is compact and we may assume that $x_n\to x_0$ in $\widetilde X$ as $n\to\infty$. 
	We claim that $x_0$ is a fixed point. For, fix $t\in\R$. 
	If there are infinitely many $n\in\N$ such that $t=j_n t_n$ for some $j_n\in\Z$, then
	$\phi_t(x_n)=\phi_{j_n t_n}(x_n)=\phi_{t_n}(x_n)=x_n$. Passing to the limit $n\to\infty$ 
	along a subsequence, it follows that $\phi_t(x_0)=x_0$. Hence it is no loss of generality 
	to assume that $t\neq j t_n$ for all $n\in\N$ and $j\in\Z$. For each $n\in \N$, there exists a unique $j_n\in\Z$ such that
	$j_n-1<t/t_n<j_n$. Then $0<j_n t_n-t<t_n$, and thus $j_n t_n\to t$ as $n\to\infty$, owing to $t_n\to 0$ as $n\to\infty$.
	As a result, $\phi_{j_n t_n}(x_n)=\phi_{t_n}(x_n)=x_n$
	yields $\phi_t(x_0)=x_0$ in the limit $n\to\infty$. 
	We deduce that $x_0\in\widetilde{X}$ is a fixed point, which is impossible. 
\end{proof}

\begin{remark}\rm 
	Note that the converse of the above lemma is not true. As we will see in Proposition \ref{sfl} below, a separating flow does not have
	periodic points with arbitrarily small periods, but the set of fixed points may not be open; see Example \ref{khex2}.
\end{remark}
The proof of Proposition \ref{fspn} has shown the following property.

\begin{corollary}
	If a continuous flow  on a compact metric space has periodic orbits of arbitrarily 
	small periods, then the flow must have a fixed point. 
\end{corollary}
The following result was rigorously proved in \cite[Lemma 2.1]{artigue18}. 
\begin{lemma}\label{equilm} 
	The following assertions are equivalent for a continuous flow $\phi$ on a compact metric space.
	
	(i) ${\rm fix}(\phi)$ is an open set.
	
	(ii) There is $T_*>0$ such that for all $T\in (0,T_*)$ there is $\xi>0$ such that
	$d(\phi_T(x),x)>\xi$ for all $x\notin {\rm fix}(\phi)$. 
\end{lemma}
%

All classes of expansive flows we are going to present have finitely many fixed points. 
The next result comes from the theory of general topology.

\begin{lemma}\label{fiso}
	Let $(X,d)$ be a metric space and let $A\subset X$ be a finite set. 
	Then $A$ is open if and only if  each point in $A$ is an isolated point of $X$.
\end{lemma}
\begin{proof} For $\eps>0$ and $x\in X$, denote by $B_\eps(x)$ the open ball of radius $\eps$ centered at $x$. 
	Let $A\subset X$ be finite.  Define $\rho_A=\min \{d(x,y), x,y\in A, x\ne y\}>0$. 
	Suppose that $A$ is open, then for any $x\in A$ there is $\eps>0$ such that $B_\eps(x)\subset A$.
	If we take $0<\delta<\min\{\rho_A,\eps\}$, then $B_\delta(x)=\{x\}$ implies that $x$ is an isolated point of $X$.
	Conversely, if any $x\in A$ is an isolated point of $X$, then there exists $B_\eps(x)=\{x\}\subset A$
	shows that $A$ is open. The lemma is proved. 	
\end{proof}

\subsection{Fixed points of suspension flows}

Let $\sigma: X\to X$ be a homeomorphism and $f: X\to [0,\infty)$  a continuous function. Set
\[X^f=\{(x,s): 0\leq s< f(x),x\in X \}\subset X \times [0,\infty) .\]
Define a new space $X(\sigma,f)=X^f/\sim$, 
where $\sim$ is the identification $(x,f(x))\sim (\sigma(x),0)$.
A metric on $X(\sigma, f)$ introduced by Bowen and Walters in \cite{bw} makes $X(\sigma,f)$ a compact metric space.

\begin{figure}[ht]
	\begin{center}
		\begin{minipage}{0.5\linewidth}
			\centering
			\includegraphics[angle=0,width=0.7\linewidth]{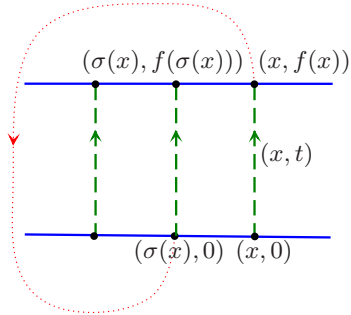}
		\end{minipage}
	\end{center}
	\caption{Suspension flow}\label{sf}
\end{figure}

\begin{definition} \label{susdf}\rm
	The {\em suspension of $\sigma$ under $f$} is the flow $\sus^{\sigma,f}_t:X(\sigma,f)\rightarrow X(\sigma,f)$
	defined  by 
	\[\sus_t^{\sigma,f}(x,s)= \Big(\sigma^k (x), t+s\pm\sum_{i=1}^{k-1} f(\sigma^i (x)) \Big), \]
	where $k$ is the unique number satisfying $0\le t+s\pm\sum_{i=0}^{k-1}f(\sigma^i (x)) < f(\sigma^k (x))$;
	see Figure \ref{sf} for an illustration.
\end{definition}

\begin{proposition}\label{sfp}\it
	Any suspension flow has no fixed points.
\end{proposition}
\begin{proof}
	By contradiction suppose $(x,s)\in X(\sigma,f)$ is a fixed point of $\sus^{\sigma,f}$. Then 
	$\sus^{\sigma,f}_t(x,s)=(x,s)$ implies that 
	\begin{equation}\label{tk}
		t= \pm (f(x)+f(\sigma(x))+\cdots + f(\sigma^{k-1}(x))).
	\end{equation} 
	This means that for any $t\in\R$, there exists $k\in \Z$ satisfying \eqref{tk}.
	But the set $$\{f(x)+f(\sigma(x))+\cdots + f(\sigma^{k-1}(x)), k\in \Z \}$$ is countable,
	which is impossible. 
\end{proof}

\section{Varieties of expansivity} \label{s3}
In this section we recall the definitions of separating and expansive homeomorphisms and flows
and present properties of several varieties of  expansive flows. 
\subsection{Separating and expansive homeomorphisms}

\begin{definition}\rm 
	Let $\sigma: X\to X$ be a homeomorphism.
	
	(i) $\sigma$ is called \textit{separating} if there exists $\delta>0$ such that if $x,y\in X$ and $d(\sigma^n(x),\sigma^n(y))<\delta$, then 
	$x=\sigma^k(y)$ for some $k\in\Z$.
	
	(ii) $\sigma$ is called \textit{expansive} if there exists $\delta>0$ such that if $x,y\in X$ and $d(\sigma^n(x),\sigma^n(y))<\delta$, then 
	$x=y$. 
	
\end{definition}
It is clear that an expansive homeomorphism is separating. Next 
we introduce a homeomorphism which is separating but not expansive. Recall \cite[Example 2.24]{artigue16} with an addition in order to $\sigma$ be well-defined. 
\begin{example}\rm \label{he} Consider a closed subset of the sphere $\R^2\cup\infty$
	\[X=\{\infty\}\cup\{(p,0): p\in\Z\}\cup\{(p,\pm 1/q): p\in\Z, q\in\Z^{+}, |p|\leq q \}.\]
	Define a homeomorphism $\sigma: X\to X$ by $\sigma(\infty)=\infty, \sigma(p,0)=(p+1,0),
	\sigma(p,\pm\frac{1}q)=(p+1,\pm\frac{1}q)$ if $p<q$, $\sigma(q,\pm\frac{1}q)=(-q,\mp\frac{1}q)$,
	$\sigma(-q,\pm\frac{1}{q})=(-q+1,\pm\frac{1}{q})$. 
	A short calculation shows that for all $0<\delta<1$,
	only $x=(0,\frac1q)$ and $y=(0,-\frac1q)$ with $p>2/\delta$ satisfy
	\[ d(\sigma^n(x),\sigma^n(y))=\frac{2}{q}<\delta\quad\mbox{for all}\quad n\in\Z.\]
	Furthermore, $\sigma^{2q+1}(0,\frac1q)=(0,-\frac1q)$
	implies $x$ and $y$ are in the same orbit (see Figure \ref{susfg}) and thus
	$\sigma$ is separating.
	However, due to $x\ne y$, it follows that $\sigma$ is not expansive. 
\end{example}

\begin{figure}[h] 
	\begin{center}
		\begin{minipage}[t]{0.6\linewidth}
			\centering
			\includegraphics[angle=0,width=1\linewidth]{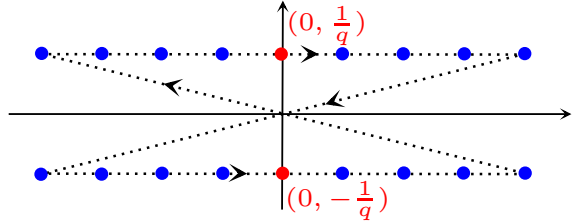}
		\end{minipage}
	\end{center}
	\caption{$(0,1/q)$ and $(0,-1/q)$ are in the same orbit}\label{susfg}
\end{figure}
\begin{example}[Shift map]  \rm 
	Consider $X=\{0,1\}^\Z=\{ (x_n)_{n\in\Z}: x_n\in\{0,1\}\}$
	with metric $d(x,y)=\sum_{n=-\infty}^\infty\frac{|x_n-y_n|}{2^{|n|}}$ for $x=(x_n), y=(y_n)\in X$.
	$(X,d)$ is a compact metric space and the shift map $\sigma: X\to X, (\sigma(x))_n=x_{n+1}$
	is a homeomorphism. 
	If $x\ne y$, then $x_n\ne y_n$ for some $n\in \Z$ implies  $d(\sigma^n(x),\sigma^n(y)))\geq 1/2$
	and thus $\sigma$ is expansive. 
\end{example}

\subsection{$C$-expansive and $K$-expansive flows}\label{sec3}
\begin{definition}[\cite{bw,komuro}]\rm Let $(X,d)$ be a compact metric space and let $\phi:\R\times X\rightarrow X$
	be a continuous flow.
	
	(i) $\phi$ is called {\em $C$-expansive}\footnote{In \cite{bw}, it is simply called \textit{expansive}. In some literatures, it is called \textit{BW-expansive}.} if for each $\eps>0$, 
	there exists $\delta>0$ such that if $x,y\in X, s\in C$ and 
	\[d(\phi_t(x),\phi_{s(t)}(y))<\delta \quad \mbox{for all}\quad t\in \R,\]
	then $y=\phi_\tau(x)$ for some $\tau\in (-\eps,\eps)$.
	
	(ii) $\phi$ is called {\em $K$-expansive} if for each $\eps>0$, 
	there exists $\delta>0$ such that if $x,y\in X, s\in K$ and 
	\[d(\phi_t(x),\phi_{s(t)}(y))<\delta \quad \mbox{for all}\quad t\in \R\]
	then $y=\phi_\tau(x)$ for some $\tau\in (-\eps,\eps)$.
\end{definition}
Such a $\delta$ is called an \textit{expansive constant} for $\eps$. 

\begin{example} \label{varphiex}\rm (a) Anosov flows are $C$-expansive (see \cite[Appendix]{flinn}).
	
	(b)  \rm Let $\Gamma$ be a discrete subgroup of $\PSL(2,\R)=\SL(2,\R)/\{\pm E_2\}$, 	where $\SL(2,\R)$ denotes the set of real matrices $2\times 2$ with unit determinant and $E_2$ denotes the unit matrix. Assume that the space
	$X=\Gamma\backslash\PSL(2,\R)$ is compact.   	
	Define a flow $(\varphi_t)_{t\in\R}$ in $X$ by 
	$\varphi_t(\Gamma g)=\Gamma g a_t$ for all $t\in\R$,
	where $a_t=\{\pm A_t\}$, $A_t={\scriptsize\begin{pmatrix}
			e^{t/2}&0\\0& e^{-t/2}
	\end{pmatrix}}$.
	Then $\varphi$ is $C$-expansive; see \cite[Theorem 3.2]{hien19}.
	
	(c) The suspension of an expansive homeomorphism under a continuous function is a $C$-expansive flow; see \cite[Theorem 6]{bw}. 
\end{example} 

It was shown in \cite{bw} that every fixed point of a $C$-expansive flow is isolated. Therefore,
the set of fixed points is finite and open. The authors also proved that in the case of fix-point-free flows,
$C$-expansivity is equivalent to $K$-expansivity (Theorem 2). Then Oka \cite{oka90}
proved that it is also true for flows with fixed points. 
\begin{theorem}[\cite{oka90}] A flow is $C$-expansive if and only if it is $K$-expansive.
\end{theorem}
A key point in the proof of the previous theorem is 
that each fixed point of a $K$-expansive flow is isolated. 
Next we provide a shorter proof by improving the one in \cite{oka90}.
\begin{proposition}\it 
	Every fixed point of a $K$-expansive flow is an isolated point of the space. 
\end{proposition}
\begin{proof} Let $\phi;X\to X$ be a $K$-expansive flow and $\delta>0$ an expansive constant for $\eps=1$. Since $X$ is compact,  $\phi: [-2,2]\times X\to X$ is uniformly continuous. There exists $\rho>0$ such that if $x,y\in X$, $d(x,y)<\rho$, then $d(\phi_t(x),\phi_t(y))<\delta/2$ for all $t\in [-2,2]$.
	Fix $x_0\in \fix(\phi)$. We show that 
	$B_{\rho}(x_0)=\{x_0\}$, which implies that $x_0$ is an isolated point in $X$. 
	Suppose in contrary that there exists $y_0\in B_\delta(x_0)\setminus \{x_0\}$. If $d(\phi_t(y_0),\phi_t(x_0))<\delta$ for all $t\in\R$, then 
	$y_0=x_0$. Therefore 
	\begin{equation}\label{t0}
		d(\phi_{t_0}(y_0),x_0)\geq \delta\quad\mbox{for some} \quad t_0\in\R.
	\end{equation}
	We verify that $y_0$ is a periodic point with a period less than 2. 
	Let $y_1=\phi_1(y_0)$ and define
	\[s(t)=\begin{cases}
		t-1 & \mbox{if}\quad  |t|\geq 2,\\
		\frac{t}{2} &\mbox{if}\quad  0\leq t\leq 2,\\
		\frac{3t}{2}&\mbox{if}\quad  -2\leq t\leq 0.
	\end{cases} \]
	Then $s:\R\to\R$ is an increasing homeomorphism on $\R$ with $s(0)=0$. Owing to 
	$d(y_0,x_0)<\rho$, we have $d(\phi_t(y_0),\phi_t(x_0))<\delta/2$ for all $|t|\leq 2$. Hence,
	\begin{align*} d(\phi_t(y_0),\phi_{s(t)}(y_1))
		&=
		d(\phi_t(y_0),\phi_{s(t)+\eta}(y_0)) \\
		&
		\leq d(\phi_t(y_0), x_0)+d(x_0,\phi_{s(t)+\eta}(y_0))
		<  \delta\mbox{ \ for all\ } |t|\leq 2.
	\end{align*}
	In conjunction with  
	\[d(\phi_t(y_0),\phi_{s(t)}(y_1))=
	d(\phi_t(y_0),\phi_t(y_0))=0\quad \mbox{for}\quad |t|\geq 2, \]
	we obtain
	\begin{align*}
		d(\phi_t(y_0),\phi_{s(t)}(y_1))<\delta\quad\mbox{for all}\quad t\in \R.
	\end{align*} 
	By hypothesis, $y_1=\phi_\tau(y_0)$ for some $|\tau|<1$ and thus $\phi_{1-\tau}(y_0)=y_0$. This means that $y_0$ is a periodic orbit
	of $\phi$ with period $0<1-\tau<2$. 
	Using	$d(\phi_t(y_0),x_0)<\delta_0$ for all $t\in [-2,2]$, we have $d(\phi_t(y_0),x_0)<\delta_0$ for all $t\in\R$,
	which contradicts \eqref{t0}. The lemma is proved.
\end{proof} 
%
%

\subsection{$K^*$-expansive and geometric expansive flows}\label{3.3}
Let $\phi:X\to X$ be a continuous flow. We recall another distance in $X$ introduced by Artigue in \cite{artigue13}:
\[d_\phi(x,y)=
\begin{cases}\inf\{\diam(\phi_{[a,b]}(z)):z\in X, [a,b]\subset \R, x,y\in\phi_{[a,b]}(z) \}
	&\mbox{if}\ y\in\phi_\R(x)
	\\
	\diam(X) & \mbox{if}\ y\notin \phi_\R(x).
\end{cases}  \]

\begin{definition}[\cite{artigue16,komuro}]\rm Let $\phi: X\to X$ be a continuous flow.
	
	(i) $\phi$ is called \textit{geometric expansive}\footnote{In \cite{artigue13}, it is simply called \textit{expansive flow}.} if for each $\eps>0$, there exists $\delta>0$ 
	such that if $x,y\in X, s\in K$ satisfying
	\begin{equation*}d(\phi_t(x),\phi_{s(t)}(y))<\delta \quad \mbox{for all}\quad t\in \R,
	\end{equation*}
	then $d_\phi(x,y)<\eps$.
	
	(ii) $\phi$ is called \textit{$K^*$-expansive} if for each $\eps>0$, there exists $\delta>0$ 
	such that if $x,y\in X, s\in K$ satisfying
	\begin{equation*}d(\phi_t(x),\phi_{s(t)}(y))<\delta \quad \mbox{for all}\quad t\in \R,
	\end{equation*}
	then $\phi_{s(t_0)}(y)=\phi_{t_0+\tau}(x)$ for some $t_0\in\R$ and $\tau\in (-\eps,\eps)$.
\end{definition}

\begin{theorem}
	The flow $\phi$ is geometric expansive if and only if it is $K^*$-expansive.
\end{theorem}
See \cite[Theorem 1.3]{artigue13} for a proof.

\begin{example}[Lorenz attractor, \cite{komuro}]  \label{lorenz}\rm 
	(a) 
	The Lorenz attractor, defined as the inverse limit of a semi-flow on a 2-dimensional branched manifold, is $K^*$-expansive (equivalently, geometric expansive). 
	
	(b) The Lorenz attractor is not $K$-expansive because it has two fixed points which are non-isolated.
\end{example}
As seen in the previous example, a $K^*$-expansive flow may have non-isolated fixed points.
In the case that each fixed point is an isolated point,  $K$-expansivity and $K^*$-expansivity are equivalent:
\begin{theorem}\label{KK*}
	Suppose that $\fix(\phi)$ is open. Then  $\phi$ is $K$-expansive
	if and only if $\phi$ is $K^*$-expansive.
\end{theorem}
To prove this theorem, we need a weaker result. 
\begin{lemma}[\cite{oka90}]\label{ld}
	A fixed-point-free flow is $K$-expansive if and only if it is $K^*$-expansive.
\end{lemma}
\begin{proof}[Proof of Theorem \ref{KK*}]
	We only need to prove the reverse.
	Let $\widetilde{X}=X\setminus \fix(\phi)$ and $\psi=\phi|_{\widetilde{X}}$, i.e.,
	$\psi: \R\times\widetilde X\to \widetilde X, \psi_t(x)=\phi_t(x)$ for all $(t,x)\in\R\times\widetilde X$.
	Suppose that $\phi$ is $K^*$-expansive on $X$, then $\psi$ is $K^*$-expansive on $\widetilde X$.
	Since $\psi$ has no fixed points in $\widetilde X$, it follows that 
	$\psi$ is $K$-expansive on $\widetilde X$, owing to Lemma \ref{ld}.
	For each $\eps>0$, let $\delta_1$ be an expansive constant for $\eps$ of $\psi$. 
	Due to the fact that $\fix(\phi)$ is open and finite, each fixed point of $\phi$ is isolated; see Lemma \ref{fiso}.
	Let $\delta_2>0$ be such that $B_{\delta_2}(x)=x$ for all $x\in \fix(\phi)$.
	We show that $\delta=\min\{\delta_1,\delta_2\}$ is an expansive constant for $\eps$ of $\phi$. 
	Let $x,y\in X$ and $s\in K$ be such that 
	\[d(\phi_t(x),\phi_{s(t)}(y))<\delta\quad\mbox{for all}\quad t\in\R. \]
	If $x,y\in \widetilde X$, then $y=\psi_\tau(x)=\phi_\tau(x)$ for some $\tau\in(-\eps,\eps)$
	since $\psi$ is expansive on $\widetilde X$. 
	If either $x\in \fix(\phi)$ or $y\in\fix(\phi)$, then $d(x,y)<\delta\leq\delta_2$ implies that $x=y$.
	The proof is complete. 
\end{proof}

\subsection{Kinematic expansive flows}
\begin{definition}[\cite{artigue16}]\rm A continuous flow  $\phi$ on $X$
	is called {\em kinematic expansive}\footnote{It is called \textit{$\{id\}$-expansive} in \cite{keynes}.} if for each $\eps>0$, there exists $\delta>0$ 
	such that if $x,y\in X$,
	\begin{equation*}d(\phi_t(x),\phi_{t}(y))<\delta \quad \mbox{for all}\quad t\in \R,
	\end{equation*}
	then $y=\phi_\tau(x)$ for some $\tau\in(-\eps,\eps)$.
\end{definition}
Here is an equivalent statement of kinematic expansivity.
\begin{proposition}[\cite{artigue16}]\label{kel}\it A flow $\phi$ is kinematic expansive if and only if for all $\eps>0$
	there exists $\delta>0$ such that if $d(\phi_t(x),\phi_t(y))<\delta$ for all $t\in\R$, then
	$d_\phi(x,y)<\delta$. 
\end{proposition} 

Recall that the orbit of flow $\phi$ through $x\in X$ is defined by $\phi_\R(x)=\{ \phi_t(x), \ t\in\R \}.$
\begin{definition}[\cite{flinn}]\label{tcdn}\rm Let $X$ be a metric space and let $\phi,\psi:\R\times X\to X$ be continuous flows. We say 
	that $\phi$ is a {\em time change} of $\psi$ if 
	for every $x\in X$ the orbits $\phi_\R(x)$, $\psi_\R(x)$ 
	and their orientations coincide.
\end{definition}
A kinematic expansive flow may be not a time change invariant. The following definition is natural.
\begin{definition}[\cite{artigue16}]\rm 
	A flow is called \textit{strong kinematic expansive}\footnote{It is equivalent to the notion of \textit{weakly expansivity} in \cite{flinn}.}
	if any time change of its is kinematic expansive. 
\end{definition}

\begin{figure}[h]
	\begin{center}
		\begin{minipage}{0.6\linewidth}
			\centering
			\includegraphics[angle=0,width=0.5\linewidth]{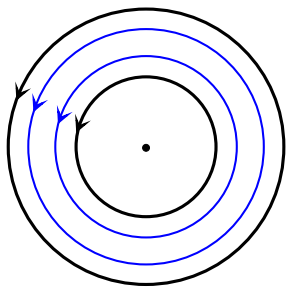}
		\end{minipage}
	\end{center}
	\caption{Periodic orbits in the annulus}\label{fig2}
\end{figure}

\begin{example}[\cite{artigue16}]\label{ccex}\rm 
	Consider  flow $\phi$ generated by the differential equation $(\dot x_1,\dot x_2)=\frac{1}{\sqrt{x_1^2+x_2^2}}(-x_2,x_1)$ on the annulus $A=\{ (x_1,x_2)\in\R^2: 1\leq x_1^2+x_2^2\leq 4\}$. For $x=(x_1,x_2)\in A$, let $\alpha=\sqrt{x_1^2+x_2^2}$ and $t_0\in [0,2\alpha\pi)$ satisfy 
	$(\alpha \cos \frac{t_0}{\alpha},\alpha\sin\frac{t_0}{\alpha})=(x_1,x_2)$. 
	Define $\phi_t(x)=(\alpha\cos \frac{t+t_0}{\alpha},\alpha\sin\frac{t+t_0}{\alpha}), t\in\R$. 
	The orbits of $\phi$ are circles $\{\alpha(\cos \frac{t}{\alpha},\sin\frac{t}{\alpha}),t\in\R\}$
	with $1\leq \alpha\leq 2$; see Figure \ref{fig2}.
	If $x,y\in A$
	are not in the same orbit of $\phi$, then for all $\delta>0$, there exists $\tau\in\R$ such that 
	$d(\phi_\tau(x),\phi_\tau(y))>\delta$ since $\phi_\R(x)$ and $\phi_\R(y)$
	are not the same periods. The kinematic expansivity of $\phi$ follows from the continuity of function $\sin^{-1}$.

	Let us consider a time change $\psi$ of $\phi$ defined as follows.
	For $x=(x_1,x_2)\in A$, choose $t_0\in[0,2\phi)$ such that
	$(x_1,x_2)=(\alpha\cos t_0, \alpha\sin t_0)$, where $\alpha=\sqrt{x_1^2+x_2^2}$. 
	Set $\psi_t(x)=\alpha(\cos (t+t_0),\sin(t+t_0)), t\in\R$. 
	Note that $\psi$ is a time change of $\phi$ due to 
	the fact that $\phi_\R(x)=\psi_\R(x)$ for all $x\in A$. 
	It is clear that $\psi$ is not separating (and hence $\phi$ is not strong separating, see Definition \ref{sdn}). Therefore, $\phi$ is kinematic expansive but not strong kinematic expansive.
\end{example}

\begin{remark} \rm It is worth mentioning that strong kinematic expansivity is more general than geometric expansivity. 
	This implies that fixed points of a (strong) kinematic expansive flow may not be isolated; see Example \ref{lorenz}. 
	Hence  the last statement in Remark 1.8.6 in \cite{fh} which claims that 
	fixed points of a kinematic expansive flow are isolated is not true; see Example \ref{khex2} for another counterexample. Kinematic expansive flows with isolated fixed points are considered in Section \ref{s4} with the name `KH-kinematic expansive flows'
\end{remark}

According to the proof of Theorem 6 in \cite{bw}, if $f: X\to [0,\infty), f(x)=1$ for all $x\in\R$,  
then it only needs the kinematic expansivity of $\sus^{\sigma,f}$ to obtain the
expansivity of homeomorphism $\sigma$. Similarly, if $f: X\to [0,\infty)$ is constant, we have
the following result. 
\begin{theorem}\label{constantf}
	Let $\sigma$ be a homeomorphism of a compact metric space $X$, and let $f: X\to [0,\infty)$ be
	constant. Then the following assertions are equivalent.
	
	(i) The homeomorphism $\sigma$ is  expansive.
	
	(ii) The suspension ${\rm sus}^{\sigma,f}$ is kinematic expansive.
	
	(iii) The suspension ${\rm sus}^{\sigma,f}$ is $C$-expansive.
\end{theorem}
\begin{remark}\rm 
	From the previous theorem, it follows that
	if a flow is the suspension of a homeomorphism under a constant function, 
	then $C$-expansive, $K^*$-expansive, strong kinematic expansive, and kinematic expansive properties 
	are equivalent.
\end{remark} 

\subsection{$C$-separating,  geometric separating and separating flows}
\begin{definition}[\cite{artigue16,gura}]\label{sdn}\rm Let $\phi: X\to X$ be a continuous flow.
	
	(i)  $\phi$ is called {\em $C$-separating} if 
	there exists $\delta>0$ such that if $x,y\in X,s\in C$ and 
	\[d(\phi_t(x),\phi_{s(t)}(y))<\delta \quad \mbox{for all}\quad t\in \R,\]
	then $y=\phi_\tau(x)$ for some $\tau\in\R$.
	
	(ii)  $\phi$ is called {\em geometric separating} if 
	there exists $\delta>0$ such that if $x,y\in X,s\in K$ and 
	\[d(\phi_t(x),\phi_{s(t)}(y))<\delta \quad \mbox{for all}\quad t\in \R,\]
	then $y=\phi_\tau(x)$ for some $\tau\in\R$.
	
	(iii)  $\phi$ is called {\em separating} if  there exists $\delta>0$  such that if
	$x,y\in X$,
	\begin{equation}\label{sdn}
		d(\phi_t(x),\phi_t(y))<\delta \quad \mbox{for all}\quad t\in \R,
	\end{equation} 
	then $y=\phi_\tau(x)$ for some $\tau\in\R$.
	
	(iv) $\phi$ is called \textit{strong separating} if every time change is separating.
	
	Such a $\delta$ in (iii) is called a \textit{separating constant} of $\phi$.
\end{definition}
It is easy to see that 
\[C\mbox{-separation } \Rightarrow \mbox{ geometric separation}\Rightarrow\mbox{ strong separation }
\Rightarrow \mbox{ separation}.\]

\begin{remark}\rm (i) According to Remark 2.17 in \cite{artigue16}, if $\phi$ is separating then every fixed point $p$ of $\phi$ is dynamically isolated, i.e. 
	there is $r>0$ such that if $d(\phi_t(x),p)<r$ for all $t\in\R$, then $x=p$. Therefore $\fix(\phi)$ is finite.
	
	(ii) Like $C$-expansive flows, each fixed point of a $C$-separating flow is 
	an isolated point of the space, whereas a fixed point of a geometric separating flow may not be isolated.
	One example is the Lorenz attractor (see \cite{komuro}), which is geometric separating
	but the fixed points are non-isolated. Therefore, 
	the Lorenz attractor is not $C$-separating and fixed points of a (strong) separating flow may not be isolated.

	(iii) It is stated in Lemma 1.8.5 in \cite{fh}  that fixed points of a separating flow  are isolated. From (ii), it follows that  the statement is not true. For another counterexample, see Example \ref{khex2}. 
	Separating flows with isolated fixed points are KH-expansive flows (see Theorem \ref{khethm}).  
\end{remark}
\begin{theorem}
	If $\fix(\phi)$ is open, then $\phi$ is $C$-separating if and only if it is geometric separating. 
\end{theorem}
\begin{proof}
	The proof is similar to that of Theorem \ref{KK*}.
\end{proof}

Although the fixed points set of a separating flow may be non-open, the property in Proposition \ref{fspn} still holds.

\begin{proposition}\label{sfl}\it
	A separating flow does not have periodic points with arbitrarily small periods.
\end{proposition}
\begin{proof} See Lemma 1.8.5 in \cite{fh} for a proof, using the fact from the proof of Proposition \ref{fspn} that if 
	periodic points $(x_n)$ with periods $(t_n)$ and $x_n\to x_0$, $t_n\to 0^+$ as $n\to\infty$, then $x_0\in\fix(\phi)$.
\end{proof}


\begin{remark}\label{srm}\rm Analogous to Theorem \ref{constantf}, in the case that a flow $\phi$ is
	the suspension of a homeomorphism under a constant time function,
	if  $\phi$ is separating, then it is $C$-separating. This holds if and only if the base homeomorphism is separating. 	
\end{remark}

%
\subsection{KH-expansive flows}

In \cite{kh} Katok and Hasselblatt introduced a definition of expansive flows, which is called KH-expansivity.
\begin{definition}\rm \label{khedn}
	Let $(X,d)$ be a compact space. 
	A continuous flow $\phi_t:X\rightarrow X$ is called {\em KH-expansive} if there exists $\delta>0$ such that if $x,y\in X, s\in C$ such that
	\begin{equation*}
		\max\{d(\phi_t(x),\phi_{s(t)}(x)),	d(\phi_t(x),\phi_{s(t)}(y))\}<\delta\quad 
		\mbox{for all}\quad t\in\R,
	\end{equation*}
	then $y=\phi_\tau(x)$ for some $\tau\in\R$. Such a $\delta$ is called a \textit{separating constant} of $\phi$. 
\end{definition}

It is obvious that $C$-expansivity implies KH-expansivity.

\begin{proposition}\label{iso}\it If $\phi$ is KH-expansive on $X$, then each fixed point of $\phi$ is an isolated point of $X$.
	Therefore,  $\fix(\phi)$ is open and finite. 
\end{proposition}
\begin{proof}
	Fix $x\in \fix(\phi)$ and let $\delta>0$
	be a separating constant of $\phi$. For $y\in B_\delta(x)$, putting $s(t)=0$ for all $t\in\R$, we have
	$d(\phi_t(x),\phi_{s(t)}(x))=d(x,x)=0$ and 
	$d(\phi_t(x), \phi_{s(t)}(y))=d(x,y)<\delta$ for all $t\in\R$. This yields $
	y=\phi_\tau(x)=x$ for some $\tau\in\R$. Consequently, $B_\delta(x)=\{x\}$ and thus $x$ is an isolated point of $X$.
	For the latter, suppose that $(x_n)\in \fix(\phi)$ and $x_n\ne x_m$ for $m\ne n$. 
	Since $X$ is compact, we may assume that $x_n\to x$ as $n\to \infty$ for some $x\in X$. 
	Owing to the continuity of $\phi$,  $\phi_t(x_n)\to \phi_t(x)$ as $n\to\infty$ for all $t\in\R$. Using $\phi_t(x_n)=x_n$ for all $n\geq 1$ and all $t\in\R$, we obtain
	$\phi_t(x)=x$ for all $t\in\R$, i.e. $x\in \fix(\phi)$ is not an isolated point of $X$, which contracts the former.
\end{proof}
\begin{theorem}\label{khethm} The following assertions are equivalent.
	
	(i) $\phi$ is KH-expansive.
	
	(ii) $\phi$ is separating and $\fix(\phi)$ is open.
	
	(iii) $\phi$ is separating and each fixed point of $\phi$ is an isolated point.
\end{theorem}
\begin{proof} 
	$(i)\Leftrightarrow (ii)$: This is \cite[Theorem 2.9]{artigue18}.
	$(ii)\Leftrightarrow (iii)$: This follows from Lemma \ref{fiso} 
	and the fact that $\fix(\phi)$ is finite.
\end{proof}
\begin{remark}\rm
	A $C$-expansive flow has finitely many periodic orbits with periods less than a given number.
	In general a KH-expansive flow does not have this property. Consider the flow in Example \ref{kex3},
	which has uncountable periodic orbits with periods smaller than $4\pi$ . However, as  separating flows, a KH-expansive flow does not have periodic orbits with arbitrarily small periods; see Proposition \ref{sfl}. 
\end{remark}
%
%
%
%
%
%
%
%

\section{KH-kinematic expansive flows}\label{s4}
In this section we introduce a new notion of expansive flows, which we call `KH-kinematic expansivity'.
Some equivalent properties are presented for this expansivity.
Since KH-kinematic expansivity is not an invariant property under time change of flows,
the concept `strong KH-kinematic expansivity' is needed.
Hierarchy of expansive flows is given with counterexamples to analyse the relations of expansive properties.

\subsection{KH-kinematic expansive flows}

The following definition is natural.

\begin{definition}\rm \label{khkedn}
	Let $(X,d)$ be a compact space. A continuous flow $\phi:\R \times X\to X$ is called {\em KH-kinematic expansive} if for every $\eps>0$, there exists $\delta>0$ such that if $x,y\in X, s\in C$ satisfying
	\begin{equation*}
		\max\{d(\phi_t(x),\phi_{s(t)}(x)),	d(\phi_t(x),\phi_{s(t)}(y))\}<\delta\quad 
		\mbox{for all}\quad t\in\R,
	\end{equation*}
	then $y=\phi_\tau(x)$ for some $\tau\in (-\eps,\eps)$.

\end{definition}

It is clear that a $C$-expansive flow is KH-kinematic expansive.

%
%

\begin{theorem}\label{khthm} Let $\phi=(\phi_t)_{t\in\R}$ be a continuous flow on $X$. Then
	the following assertions are equivalent.
	
	(i) $\phi$ is KH-kinematic expansive.
	
	(ii) $\phi$ is KH-expansive and kinematic expansive.
	
	(iii) $\phi$ is  kinematic expansive and $\fix(\phi)$ is open.
	
	(iv) $\phi$ is kinematic expansive and each $x\in\fix(\phi)$ is an isolated point of $X$.

\end{theorem}	
\noindent\textit{Proof.}  
$(i)\Rightarrow (ii)$: Suppose that $\phi$ is KH-kinematic. The former of (ii) is clear. For the latter, take $s(t)=t$ for all $t\in\R$ in Definition \ref{khkedn} to obtain the kinematic expansivity. 

$(ii)\Rightarrow (iii)$: This follows from propositions  \ref{iso}.

$(iii)\Leftrightarrow (iv)$: This is a consequence of Lemma \ref{fiso}, noting that $\fix(\phi)$ is finite if $\phi$ is kinematic expansive.

In order to show $(iii)\Rightarrow (i)$, we need the  following lemma. 
The next result states that if a  regular orbit, which is reparameterized, is close 
to the original one in the whole time, then the reparameterization must be close to the identity.
\begin{lemma}\label{deltalm} Suppose that $\fix(\phi)$ is open. 
	For each $\eps>0$, there exists $\delta>0$ such that if $x\in X\setminus \fix(\phi)$, $s\in C$
	satisfying  $d(\phi_t(x),\phi_{s(t)}(x))<\delta\quad\mbox{for all}\quad t\in\R $,
	then  $|s(t)-t|<\eps$ {for all} $t\in\R.$
\end{lemma}
\begin{proof} Let $T_*>0$  be in Lemma \ref{equilm} and let $\widetilde X=X\setminus \fix(\phi)$. 
	Suppose in contrary that there is  $0<\eps<T_*$ such that
	for $\delta_n\to 0$, there are $x_n\in \widetilde X,s_n\in C$,
	\begin{equation}\label{phis}
		d(\phi_t(x_n), \phi_{s_n(t)}(x_n))<\delta_n\quad \text{for all}\quad t\in\R 
	\end{equation} 
	but 
	\begin{equation}\label{s(t)} |s_n(t_n)-t_n|>\eps \quad\mbox{for some}\quad t_n\in \R.
	\end{equation}  
	Fix $n\in \N$ and let $h_n(t)=s_n(t)-t$ for all $t\in\R$. Then $h_n:\R\to\R$ is continuous and $h_n(0)=0$.
	It follows from \eqref{s(t)} that there exists $u_n\in (-|t_n|,|t_n|)$ such that $|h_n(u_n)|=\eps$.
	There exists a subsequence of $u_n$, which is not renumerated, such that $h_n(u_n)=\eps$ for all $n$
	or $h_n(u_n)=-\eps$ for all $n$. Without loss of generality, we may assume that $h_n(u_n)=\eps$ for all $n$. 
	Let $a_n=\phi_{u_n}(x_n)$ and $b_n=\phi_{s_n(u_n)}(x_n)$. Since $\widetilde X$ is compact, 
	we may assume that $a_n\to a\in \widetilde X$. 
	As a consequence, $b_n=\phi_{s_n(u_n)-u_n}(a_n)=\phi_{ \eps}(a_n)\to \phi_{\eps}(a)$.
	Furthermore, by \eqref{phis}, 
	\begin{equation*}
		d(\phi_{u_n}(x_n), \phi_{s_n(u_n)}(x_n))<\delta_n
	\end{equation*} 
	yields $b_n\to x$ and thus $\phi_{\eps}(a)=a$ contracting the property of $T_*$. The lemma is proved. 
\end{proof}
Now we are in a position to show  $(iii) \Rightarrow (i)$. 
Suppose that $\phi$ is kinematic expansive and $\fix(\phi)$ is open. 
Choose $\delta_1>0$ such that $B_{\delta_1}(x)=\{x\}$ for all $x\in \fix(\phi)$.
For all $\eps>0$, let $\delta_2=\delta_2(\eps)>0$ be an expansive constant for $\eps$, $\delta_3=\delta_3(\eps)$ as in Lemma \ref{deltalm}
and set $\delta=\min\{\delta_1,\delta_2,\delta_3\}/2$. Suppose $x,y\in X, s:\R\to\R$ is continuous, $s(0)=0$ such that
\begin{align}\label{kh1'}
	d(\phi_t(x),\phi_{s(t)}(x))<\delta\quad 
	\mbox{for all}\quad t\in\R, \\
	d(\phi_t(x),\phi_{s(t)}(y))<\delta \quad  \label{kh2'}
	\mbox{for all}\quad t\in\R.
\end{align} 
If $x\in \fix(\phi)$, then  $y=x$ due to  \eqref{kh2'}. 
If $x\notin\fix(\phi)$, then it follows from \eqref{kh1'} and Lemma \ref{deltalm} that 
\[|s(t)-t|<\eps \quad\mbox{for all}\quad t\in\R \]
and thus  $s:\R\to\R$ is a surjection. 
Furthermore, from  \eqref{kh1'} and \eqref{kh2'}, we have
\begin{equation}\label{stt}
	d(\phi_{s(t)}(x),\phi_{s(t)}(y))<2\delta<\delta_2\quad 
	\mbox{for all}\quad t\in\R
\end{equation}
and hence  
\[d(\phi_t(x),\phi_t(y))<\delta_2\quad \mbox{for all}\quad t\in\R  \]
is verified. 
Since $\phi$ is kinematic expansive, $y=\phi_\tau(x)$ for some $\tau\in (-\eps,\eps)$, which
shows that $\phi$ is KH-kinematic expansive. 
{\hfill$\Box$}

\begin{example}[Horocycle flow] \label{khex3}\rm  It is well-known that the horocycle on a compact surface of constant negative curvature is equivalent to a flow $\theta$ defined as follows. Recall the space $X=\Gamma\backslash\PSL(2,\R)=\{\Gamma g, g\in\PSL(2,\R)\}$ from Example \ref{varphiex}. Let
	$b_t=\{\pm B_t\}\in\PSL(2,\R)$ with $B_t={\scriptsize\begin{pmatrix}
			1&t\\0 &1
	\end{pmatrix}}\in\SL(2,\R),t\in\R$. Define a flow $\theta_t:X\to X$, $\theta_t(\Gamma g)=\Gamma g b_t$ for all $g\in\PSL(2,\R),t\in\R$. 
	According to \cite[Theorem 3.5]{hien19}, $\theta$ is kinematic expansive. Since $\theta$ has no fixed points (see  \cite[Theorem 3.9]{hien19}), it follows from Theorem \ref{khthm} that $\theta$ is  KH-kinematic expansive.
\end{example}

The next result presents another equivalent definition of KH-kinematic expansivity.
\begin{theorem}
	A flow $\phi$ is KH-kinematic expansive if and only if 
	for all $\eps>0$ there exists $\delta>0$ such that if
	$x,y\in X, s\in C$ satisfying
	\begin{equation*}
		\max\{d(\phi_t(x),\phi_{s(t)}(x)),	d(\phi_t(x),\phi_{s(t)}(y))\}<\delta\quad 
		\mbox{for all}\quad t\in\R,
	\end{equation*}
	then $d_\phi(x,y)<\eps$; recall the distance $d_\phi$ in Subsection \ref{3.3}.
\end{theorem}
\begin{proof}
	($\Rightarrow$) The proof is similar to that of \cite[Prop. 2.6]{artigue16}.
	
	($\Leftarrow$) Taking $s(t)=t$ and using Proposition \ref{kel} we deduce that
	$\phi$ is kinematic expansive. Analogously to Proposition \ref{iso},
	$\fix(\phi)$ is open and hence $\phi$ is KH-kinematic expansive, owing to Theorem \ref{khthm}.	
\end{proof}

It seems natural to ask whether we can replace the hypothesis $s\in C$ by $s\in K$ in Definition \ref{khkedn}. 
The answer is as follows.
\begin{proposition}\it Let $\phi$ be a continuous flow on $X$. The following assertions are equivalent.
	
	(i) For all $\eps>0$ there exists $\delta>0$ such that if
	$x,y\in X, s\in C$ surjective satisfying
	\begin{equation*}
		\max\{d(\phi_t(x),\phi_{s(t)}(x)),	d(\phi_t(x),\phi_{s(t)}(y))\}<\delta\quad 
		\mbox{for all}\quad t\in\R,
	\end{equation*}
	then $y=\phi_\tau(x)$ for some $\tau\in(\eps,\eps)$.
	
	(ii) $\phi$ is kinematic expansive.
\end{proposition}

\begin{proof}
	($\Rightarrow$) This is clear. 
	
	($\Leftarrow$) Suppose that $\phi$ is kinematic expansive. Let $\eps>0$ and $\delta$ a separating constant for $\eps$. 
	Set $\rho=\delta/2$. For $x,y\in X, s\in K$ such that 
	\begin{equation*}
		\max\{d(\phi_t(x),\phi_{s(t)}(x)),	d(\phi_t(x),\phi_{s(t)}(y))\}<\rho\quad 
		\mbox{for all}\quad t\in\R,
	\end{equation*}
then \[d(\phi_{s(t)}(x),\phi_{s(t)}(y))<2\rho=\delta\ \mbox{ for all }\ t\in \R .\]
	Since $s$ is surjective, 
	\[d(\phi_{t}(x),\phi_{t}(y))<\delta\ \mbox{ for all }\ t\in \R \]
	is verified. This implies $y=\phi_\tau(x)$ for some $|\tau|<\eps$. The proof is complete.	
\end{proof}

\begin{remark}\rm Analogously, if we set $s\in C$  surjective  in Definition \ref{khedn}, the definition is equivalent to that of separating flows. 
\end{remark}
\subsection{Strong KH-kinematic expansive flows}
Like kinematic expansivity, KH-kinematic expansivity is not an invariant under time change. Otherwise, KH-kinematic expansivity
would imply strong kinematic expansivity, which is not always true; see Table \ref{cte2}.
This motivates to consider the following expansivity.
\begin{definition}\rm We say that $\phi$ is {\em strong KH-kinematic expansive}
	if every time change of $\phi$ is   KH-kinematic expansive.
\end{definition}
Strong KH-expansivity is defined analogously.
The next result follows immediately from theorems \ref{khethm} and \ref{khthm}.
\begin{theorem}
	A flow is strong KH-kinematic expansive (resp. strong KH-expansive) if and only if 
	it is strong kinematic expansive (resp. strong separating)  and the set of its fixed points is open.
\end{theorem}

\begin{example}\rm \label{khgex}Recall the horocycle flow $\theta$ from Example \ref{khex3}. 
	
	(a)  The flow $\theta$  is strong KH-kinematic expansive. This can be proved similarly to  its strong kinematic expansivity in \cite[Theorem 3.9]{hien20}, or  this follows from the fact that $\theta$ is strong kinematic expansive  and the set of fixed points is empty. 
	
	(b) The flow $\theta$ is not geometric separating; see \cite[Remark 3.6(a)]{hien19}.
	Therefore it is not geometric expansive as well.
\end{example}

\begin{example}\label{kex3}\rm 
	Recall the flow $\phi$ in Example \ref{ccex}.
	It is clear that $\fix(\phi)=\varnothing$. Since $\phi$ is kinematic expansive, 
	$\phi$ is KH-kinematic expansive (and hence KH-expansive). 
	However, it is neither strong KH-expansive nor strong KH-kinematic expansive,
	which is due to the fact that it is not strong separating.
\end{example}
\subsection{KH-positive kinematic expansive flows}

Let us consider a new notion of expansive flows.
\begin{definition}\rm \label{khps}
	Let $(X,d)$ be a compact metric space. A continuous flow $\phi:\R \times X\to X$ is called \textit{KH-positive kinematic expansive} if for every $\eps>0$, there exists $\delta>0$ such that for $x,y\in X, s:[0,\infty)\to[0,\infty)$  continuous, $s(0)=0$ satisfying
	\begin{equation*}
		\max\{d(\phi_t(x),\phi_{s(t)}(x)),	d(\phi_t(x),\phi_{s(t)}(y))\}<\delta\quad 
		\mbox{for all}\quad t\in [0,\infty),
	\end{equation*} 
	then $y=\phi_\tau(x)$ for some $\tau\in (-\eps,\eps)$. 	
\end{definition}

The definitions of \textit{positive kinematic expansive flows} and \textit{KH-positive expansive flows} are analogous.

%
%

The following theorem is proved similarly to Theorem \ref{khthm}.
\begin{theorem}\label{pkhkthm} Let $\phi$ be a continuous flow on compact metric space $X$.
	Then $\phi$ is KH-positive kinematic expansive (reps. KH-positive expansive)
	if and only if $\phi$ is positive kinematic expansive (resp. separating) and $\fix(\phi)$ is open.
\end{theorem}
\begin{proof}Note that for $x\in X$ and $t\in\R$, if $\phi_t(x)=x$, then $\phi_{-t}(x)=x$.
	This yields that  if $\phi_t(x)=x$ for all $t\in [0,\infty)$, then $x$ is a fixed point of $\phi$.  
	The proof of this theorem is analogous to that of \cite[Prop. 2.9]{artigue18} and  Theorem \ref{khthm}. 
\end{proof}

The next result is a characterization of KH-positive expansive flows on compact surfaces. 
\begin{proposition}\it
	Let $\phi$ be a continuous flow on a compact surface. Then $\phi$ is KH-positive kinematic expansive
	if and only if $\phi$ is positive kinematic expansive.
\end{proposition}
\begin{proof}
	It was shown in \cite[Prop. 6.8]{artigue16} that if $\phi$ is positive expansive on a compact surface,
	then $\fix(\phi)=\varnothing$. The proposition follows from Theorem \ref{pkhkthm}.
\end{proof} 

\begin{remark}\rm A $C$-positive kinematic expansive flow has at least one periodic orbit. More clearly,  the full space is the union of finite periodic orbits and fixed points (see Lemma 4.1 and Theorem 4.2 in \cite{artigue13}). In general, a KH-positive kinematic expansive flow may not have this property. One example is the horocycle flow $\theta$ in Example \ref{khex3} which has no periodic points and no fixed points. 
\end{remark}

Let $\phi$ be a continuous flow on a compact metric space $X$ and define the inverse
flow $\phi^{-1}$ as $\phi^{-1}_t=\phi_{-t}$.

\begin{definition}\rm 
	We say that $\phi$ is \textit{KH-kinematic bi-expansive} (resp. \textit{kinematic bi-expansive}) if $\phi$ and $\phi^{-1}$ are KH-positive kinematic expansive (resp. \textit{positive kinematic expansive}). 
\end{definition}

\begin{example} \rm 	
	(a) The horocyle flow $\theta$ defined in Example \ref{khex3}
	is KH-bi-kinematic expansive. The proof is analogous to that of \cite[Theorem 3.5]{hien19}.
	
	(b) Consider the flow $\varphi$ in Example \ref{varphiex} (b). It was shown in \cite[Remark 3.6(b)]{hien19} that both $\varphi$ and $\varphi^{-1}$ are not positive separating, hence they are not KH-positive expansive.
\end{example}

\begin{lemma}
	Each fixed point of a kinematic bi-expansive flow is an isolated point of the space. 
	Therefore, the set of its fixed points is open.
\end{lemma}
For a proof, see \cite[Prop. 6.16]{artigue16}. Using the above lemma and Theorem \ref{pkhkthm} 
we obtain the following result.
\begin{theorem}
	A flow is KH-kinematic bi-expansive if and only if 
	it is kinematic bi-expansive.
\end{theorem}

\subsection{Hierarchy of expansive flows}

The hierarchy of expansive  flows is illustrated in Table \ref{tablaExp}. 

\begin{table}[h]
	\[
	\begin{array}{ccc}
		\hbox{\fbox{\quad $C$-expansivity\quad }} &  \Longrightarrow & \hbox{\fbox{\ $C$-separation \ }} \\
		\big\Downarrow && \big\Downarrow\\
		\hbox{\fbox{Strong KH-kinematic expansivity}} & \Longrightarrow & \hbox{\fbox{\ Strong KH-expansivity\ }}\\
		\big\Downarrow &&\big\Downarrow \\
		\hbox{\fbox{KH-kinematic expansivity}} & \Longrightarrow & \hbox{\fbox{\ KH-expansivity\ }}\\
		\big\Downarrow && \big\Downarrow\\
		\hbox{\fbox{Kinematic expansivity}} & \Longrightarrow & \hbox{\fbox{ \quad Separation\quad }}\\
	\end{array}
	\]
	\caption{Hierarchy of expansive flows}
	\label{tablaExp}
\end{table}

\begin{figure}[ht]
	\begin{center}
		\begin{minipage}{0.5\linewidth}
			\centering
			\includegraphics[angle=0,width=0.5\linewidth]{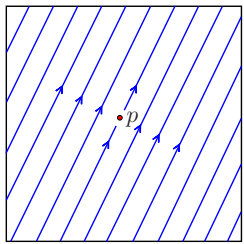}
		\end{minipage}
	\end{center}
	\caption{Strong kinematic expansive flow in the 2-torus with a fake saddle}\label{if}
\end{figure}

Let us consider some examples.

\begin{example}\label{khex2} \rm Consider an irrational flow on the 2-torus ${\mathbb T}^2=\R^2/\Z^2$ with velocity field $X$. 
	We take any non-negative smooth function $f$ with
	just one zero at some point $p$ in the torus. Denote by $\phi$ the flow generated by the vector
	field $f X$; see Figure \ref{if} for an illustration. Then $\phi$ is (strong) kinematic expansive; see \cite[Example 2.8]{artigue16}. However, due to ${\rm fix}(\phi)=\{p\}$, which is a  non-open set,
	$\phi$ is not KH-kinematic expansive. 
\end{example}

\begin{example}\label{abcex} \rm Consider the separating but not expansive homeomorphism $\sigma$ on $X$
	in Example \ref{he}. Let $f: X\to [0,\infty)$ be a constant function.
	By Theorem \ref{constantf} and Remark \ref{srm}, $\sus^{\sigma,f}$ is $C$-separating but not kinematic expansive
	and hence it is not KH-kinematic expansive. 
	In addition, due to Proposition \ref{sfp} $\sus^{\sigma,f}$ has no fixed points.
	This shows that $\sus^{\sigma,f}$ is strong KH-expansive.  
\end{example}
Table \ref{cte1} below recalls counterexamples for the hierarchy in Table \ref{tablaExp}.
\begin{table}[h!]\noindent
	\[
	\begin{array}{ccc}
		\hbox{\fbox{\quad $C$-expansivity\quad }}        & \begin{array}{c}\centernot\Longleftarrow\\ \hbox{Ex. \ref{abcex} }  \end{array} &  \hbox{\fbox{\ $C$-separation \ }}  \\
		\big\nUparrow \hbox{\hskip-0.25cm Ex. \ref{khgex} } && \big\nUparrow \hbox{\hskip-0.25cm Ex. \ref{khgex} }\\
		\hbox{\fbox{Strong KH-kinematic expansivity}} & \begin{array}{c}\centernot\Longleftarrow\\ \hbox{Ex. \ref{abcex} } \end{array} & \hbox{\fbox{Strong KH-expansivity}}\\
		\big\nUparrow \hbox{\hskip-0.25cm Ex. \ref{kex3} } && \big\nUparrow \hbox{\hskip-0.25cm Ex. \ref{kex3} }\\
		\hbox{\fbox{KH-kinematic expansivity}} & \begin{array}{c}\centernot\Longleftarrow\\ \hbox{Ex. \ref{abcex} } \end{array} & \hbox{\fbox{\ KH-expansivity\ }}\\
		\big\nUparrow \hbox{\hskip-0.25cm Ex. \ref{khex2} } && \big\nUparrow \hbox{\hskip-0.25cm Ex. \ref{khex2} }\\
		\hbox{\fbox{ Kinematic expansivity }}        & \begin{array}{c}\centernot\Longleftarrow\\ \hbox{Ex. \ref{abcex}} \end{array} & \hbox{\fbox{\qquad Separation\qquad}}\\
	\end{array}
	\]
	\caption{Diagram of counterexamples 1}
	\label{cte1}	
\end{table}
\begin{remark}
	\rm (i) All expansive (and separating) flows introduced in this paper have finite sets of fixed points. 
	Due to Lemma \ref{iso}, the set of fixed points is open if and only if each fixed point is an isolated point.
	The classes of expansive and separating flows having open fixed points sets consist of
	$C$-expansive flows, (strong) KH-kinematic expansive flows, (strong) KH-expansive flows and $C$-separating flows.

	(ii) Expansive properties   which are invariant properties under time change of flows include
	$C$-expansivity, geometric expansivity, strong KH-kinematic expansivity, strong KH-expansivity, 
	strong kinematic expansivity, $C$-separation,  geometric separation, and strong separation. 
	
	(iii) Expansive flows having kinematic expansive properties are $C$-expansive, 
	geometric expansive, (strong) KH-kinematic expansive, and (strong) kinematic expansive flows.
	
\end{remark}
The next table recalls counterexamples of some other expansive and separating flows.

\begin{table}[h!]
	\[\begin{array}{ccc}
		\hbox{\fbox{Geometric expansivity}} &\begin{array}{c}\centernot\Longrightarrow\\ \hbox{Ex. \ref{lorenz}}
		\end{array}  & \hbox{\fbox{KH-expansivity}} \\
		\mbox{Ex. \ref{lorenz}}\ \big\nDownarrow\hskip-0.20cm\big\nUparrow\hbox{\hskip-0.25cm Ex. \ref{khgex}} &&
		\mbox{Ex. \ref{kex3}}\ \big\nDownarrow\hskip-0.20cm\big\nUparrow\hbox{\hskip-0.25cm Ex. \ref{abcex}} \\
		\hbox{\fbox{Strong KH-kinematic expansivity}} & 
		\begin{array}{c}\centernot\Longleftarrow\\ \hbox{Ex. \ref{abcex}}
		\end{array} 
		& \hbox{\fbox{Strong separation}}\\
		\mbox{Ex. \ref{khgex}\ } \big\nDownarrow\hskip-0.2cm\big\nUparrow\hbox{\hskip-0.25cm Ex. \ref{abcex}} && \mbox{Ex. \ref{abcex}\ } \big\nDownarrow\hskip-0.2cm \big\nUparrow\hbox{\hskip-0.25cm Ex. \ref{kex3}}\\
		\hbox{\fbox{ $C$-separation}} & \begin{array}{c}\centernot\Longleftarrow\\ \hbox{\hskip-0.25cm Ex. \ref{khex2} } \end{array} & \hbox{\fbox{ Kinematic expansivity}}
	\end{array}\]
	\caption{Diagram of counterexamples 2}
	\label{cte2}
\end{table}
The counterexamples are based on the property of fixed points sets, the strong and `kinematic' properties (see previous remark).
Let us explain the left column in Table \ref{cte2}.
Due to the fact that the fixed points set of a (strong) KH-kinematic expansive is open, whereas that of a geometric expansive flow may be non-open,
one of KH-kinematic expansivity and geometric expansivity does not imply the other. 
The same occurs to strong KH-kinematic expansivity and $C$-separation, which is since a strong KH-kinematic expansive flow may not a conjugacy invariant
and a $C$-separating flow may not have kinematic expansivity. 
The right column in Table \ref{cte2} can be explained analogously.

\noindent\textbf{Summary of counterexamples:}
We have presented several  variations of expansive and separating flows on compact metric spaces. In Table \ref{cte1} we recall the counterexamples in the hierarchy in Table \ref{tablaExp}. The  concepts of expansivity and separation in Table \ref{cte1} are not equivalent  in the general context of continuous flows on compact metric spaces. 
Table \ref{cte2} recalls the counterexamples in the hierarchy of some other expansive and separating properties. 
These may be seen as a supplement of the hierarchy provided by Artigue in \cite{artigue16}. 

The following questions arise naturally.

\noindent\textbf{Question 1.} \textit{Do geometric separation and kinematic expansivity imply geometric expansivity?}

\noindent\textbf{Question 2.} \textit{Do strong separation and kinematic expansivity imply strong kinematic expansivity?}

\noindent{\bf Acknowledgments.}  This research is funded by Vietnam National Foundation for Science and Technology Development (NAFOSTED) under grant number 101.02-2020.21. This paper was done during the author's stay at Vietnam Institute
for Advanced Study in Mathematics (VIASM). He would like to thank VIASM for its wonderful
working condition.   The author enjoyed many fruitful discussions with  Nguyen Bao Tran.

\end{document}